\documentclass[11pt,reqno]{amsart}
\usepackage{amssymb}
\usepackage{amscd,enumerate,amsfonts,calc,amsmath,verbatim,xypic}
\usepackage[alphabetic]{amsrefs}
\usepackage{amsmath, amssymb, amsfonts, amstext, amsthm, amscd, xypic}
\usepackage[mathscr]{eucal}
\usepackage{enumerate}
\usepackage{tikz,color,soul}
\usetikzlibrary{matrix,arrows}
\usepackage[utf8]{inputenc}
\usepackage[english]{babel}

\usepackage{mathrsfs}

\usepackage{verbatim}
\usepackage{blindtext}
\usepackage{geometry}
\usepackage[new]{old-arrows}

\usepackage{amsmath,calligra,mathrsfs}
\DeclareMathOperator{\Hom}{\mathscr{H}\text{\kern -3pt {\calligra\large om}}\,} 

\usepackage{longtable}
\usepackage{hyperref}
\usepackage[all]{xy}
\usepackage{ulem}
\usepackage{xcolor}

\usetikzlibrary{arrows.meta, positioning,arrows}

\newtheorem{thm}{Theorem}[section]
\newtheorem{prop}[thm]{Proposition}
\newtheorem{lemma}[thm]{Lemma}
\newtheorem{cor}[thm]{Corollary}
\newtheorem{defn}[thm]{Definition}

\newtheorem{rmk}[thm]{Remark}

\numberwithin{equation}{section}

\def\rar{{\rightarrow}}

\def \mc{\mathcal}

\allowdisplaybreaks
\begin{document}
	
	\title{$G$-connections on principal bundles over complete $G$-varieties}

	\author{Bivas Khan}
	\address{Department of Mathematics, Chennai Mathematical Institute, Chennai, India}
	\email{bivaskhan10@gmail.com}

	\author{Mainak Poddar}
	\address{Department of Mathematics, Indian Institute of Science Education and Research-Pune, Pune, India}
	\email{mainak@iiserpune.ac.in}



	\begin{abstract}
		Let \(X\) be a complete variety over an algebraically closed field $k$ of characteristic zero, equipped with an action of an algebraic group \(G\). Let \(H\) be a reductive group. We study the notion of \(G\)-connection on a principal \(H\)-bundle. We give necessary and sufficient criteria for the existence of \(G\)-connections extending the Atiyah-Weil type criterion for holomorphic connections obtained by Azad and Biswas.	
		 We also establish a relationship between the existence of \(G\)-connection and equivariant structure on a principal \(H\)-bundle, under the assumption that \(G\) is semisimple and simply connected. These results have been obtained by Biswas et al. 
		 when the underlying variety is smooth.
		
	\end{abstract}
	
	\maketitle
	
	
	\section{Introduction}
	Atiyah initiated the study of holomorphic connection for holomorphic principal bundles (\cite{At57}). However, the existence
	of a connection in the holomorphic category is not guaranteed. A well-known theorem due to A. Weil and M. F. Atiyah says that a holomorphic vector bundle $E$ over a compact connected Riemann surface $X$ admits a flat connection if and only if $\text{degree}(W) = 0$ for every direct summand $W$ of $E$. This criterion has been extended for principal bundles over $X$ having reductive structure group in \cite{AB02, AB03}.
	
	When a manifold \(X\) admits an action of a Lie group \(G\), the notion of \(G\)-connection on principal bundles over \(X\) was introduced in \cite{BSN15}. The authors in \cite{BSN15, BPeqcn} investigated various necessary and sufficient conditions under which a principal bundle over $X$ admits a $G$-connection. If \(X\) is a toric variety under the action of a torus \(T\), then a \(T\)-connection is nothing but the logarithmic connection singular along the boundary divisor \(D:=X \setminus T\) (see \cite[Section 4.2]{BSN15}). In \cite{DKP}, the authors have studied logarithmic connections on principal bundles over $X$, which is singular along $D$, where \(X\) is a normal projective variety (possibly with singularities) and \(D\) is a reduced Weil divisor on $X$. The authors have shown that  the existence of a logarithmic connection on a principal bundle over a projective toric variety, singular along the boundary divisor, is equivalent to the	existence of a torus equivariant structure on the bundle. The aim of this note is to consider \(G\)-connection, when \(X\) is a complete variety equipped with an action of more general algebraic group \(G\) than that of an algebraic torus.  Then our goal is to establish an analogue of Atiyah-Weil type criterion for existence of \(G\)-connection.

	Let \(X\) be a complete variety over an algebraically closed field \(k\) of characteristic zero. Let \(G\) be an algebraic group acting on \(X\). Let \(H\) be a reductive linear algebraic group and \(p: \mathcal{P} \rightarrow X\) be a principal \(H\)-bundle. We present an algebro-geometric construction of the \(G\)-Atiyah sequence introduced in \cite{BSN15} (see Section \ref{Atiyah sequence for group action}). This will enable us to define a \(G\)-connection as an $\mathcal{O}$-linear splitting of the \(G\)-Atiyah sequence. We then study relationship between \(G\)-connection on a vector bundle and its associated frame bundle (see Section \ref{Gconnection on vector bundle}). We provide several necessary and sufficient conditions for existence of \(G\)-connection in the spirit of results obtained in \cite{AB02} and \cite{AB03}, in terms of the connections on associated bundles.  We show that to check the existence of \(G\)-connection it suffices to consider Levi reduction (see Corollary \ref{GconnLeviRed}). When \(H\) is semisimple, we show that a principal \(H\)-bundle admits a (flat) $G$-connection if and only if the associated adjoint bundle admits one (see Corollary \ref{exiat_ad}). More generally, for a reductive group \(H\), the necessary and sufficient condition on a $H$-bundle to admit a flat
	connection is described  (see Proposition \ref{Thm3.1AB3}):
	\begin{thm}
			Let \(H\) be a reductive linear algebraic group. Then a principal \(H\)-bundle $\mathcal{P}$ admits a \(G\)-connection if and only if the following conditions hold:
		\begin{enumerate}
			\item the adjoint bundle $\text{ad}(\mathcal{P})$ admits a \(G\)-connection,
			\item for every character $\chi$ of \(H\), the associated line bundle $\mathcal{P} \times^{\chi} k$ associated to $\mathcal{P}$ admits a \(G\)-connection. 
		\end{enumerate}
	\end{thm}

	 In Section \ref{infi_def}, we characterize the existence of a \(G\)-connection using the infinitesimal deformation of the principal bundle. Section \ref{Equivariant structure and Gconnection} is devoted to investigating relationship between \(G\)-connection and $G$-equivariant structure on a principal bundle. We show the following (see Corollary \ref{exist_conn_1}):
	\begin{thm}
	Let \(p : \mathcal{P} \rightarrow X\) be a principal \(H\)-bundle. Assume that \(G\) is semisimple and simply connected. Then the following conditions are equivalent.
	\begin{enumerate}
		\item The principal bundle admits an equivariant structure.
		\item The principal bundle admits an integrable \(G\)-connection.
		\item The principal bundle admits a \(G\)-connection.
	\end{enumerate} 
		\end{thm} 
	Finally, as an application we show that a principal \(H\)-bundle on \(X\) admits a \(G\)-equivariant structure if and only if  for all \(g \in G\), we have
	\begin{equation*}
		\Phi_g : \mathcal{P} \stackrel{\cong} \longrightarrow \rho_g^* \, \mathcal{P}
	\end{equation*}
	as principal \(H\)-bundles over \(X\), where \(G\) is assumed to be semisimple and simply connected (see Proposition \ref{klytype}). This generalizes a well-known criterion for existence of torus equivariant structure on bundles (cf. \cite[Proposition 1.2.1]{Kly}, \cite[Corollary 4.4]{BPeqcn}). In the final section we illustrate some natural \(G\)-connections with examples.

		\subsection*{Acknowledgments} The first author thanks Vikraman Balaji for helpful conversations. The first author is partially supported by a grant from Infosys Foundation. The research of the second author was supported in part by a SERB MATRICS Grant: MTR/2019/001613.

	\section{Preliminaries}
	
 Throughout this article, we consider algebraic varieties, 	schemes, and morphisms over an algebraically closed field $k$ of characteristic zero. Unless explicitly mentioned, we will assume that the considered schemes are of finite type over $k$. By a point of a scheme $X$, we will mean a closed point unless explicitly mentioned. A variety is an integral separated scheme. By algebraic group we mean smooth group scheme over $k$. Since, $k$ is of zero characteristic, any algebraic group is reduced and hence, a group variety. By linear algebraic group we mean affine algebraic group. 
 
 Let \(X\) be a variety and \(H\) be  a linear algebraic group. A principal \(H\)-bundle is a faithfully flat \(H\)-invariant morphism \(p : \mathcal{P} \rightarrow X\) from a right \(H\)-variety $\mathcal{P}$ to \(X\) such that the morphism 
 \begin{equation}\label{defpb}
 	\mathcal{P} \times H \rightarrow \mathcal{P} \times_X \mathcal{P}, \, (z, h) \mapsto (z \cdot h, z)
 \end{equation} 
	is an isomorphism.
	\begin{rmk}\cite[Remark 3.1]{Brauto}
		The principal bundle \(p : \mathcal{P} \rightarrow X\) is locally isotrivial, i.e. for any point \(x \in X\) there exists an open subscheme \(U \subset X\) containing \(x\) and a finite \'{e}tale surjective morphism \(f : U' \rightarrow U\) such that the pullback bundle \(\mathcal{P} \times_U U'\) is trivial.
	\end{rmk}
	
We now recall the automorphism functor from \cite[Section 3]{MO}. Given a scheme \(S\), we denote by \(\text{Aut}_S(X \times S)\) the group of automorphisms of \(X \times S\) viewed as  a scheme over \(S\). The automorphism functor is given by the following contravariant group functor from the category $\text{Sch}/k$ of schemes over $k$ to the category Gr of groups.
\begin{equation*}
	\underline{Aut}(X) : \text{Sch}/k \longrightarrow \text{Gr}, \, S \longmapsto \text{Aut}_S(X \times S).
\end{equation*}
If this functor is representable we say that the automorphism group scheme of \(X\) exists. 	Assume that \(X\) is proper scheme over \(k\), then the functor \(\underline{Aut}(X)\) is representable by \(\text{Aut}(X)\), a group scheme locally of finite type over \(k\) (see  \cite[Theorem 3.7]{MO}). In particular, the neutral component \(\text{Aut}^0(X)\) is an algebraic group. Also, recall that we have the following isomorphism of Lie algebras (see \cite[Lemma 3.4 and Introduction]{MO}, \cite[Theorem 3.1]{LI10}, \cite[Section 2]{Br18}). Note that here we need characteristic \(k\) to be zero.
\begin{equation}\label{P1}
\begin{split}
		\alpha_X : \text{Lie}(\text{Aut}(X)) & \stackrel{\cong}\longrightarrow H^0(X, \mathscr{T}_X), \, \text{given by}\\
		\phi & \longmapsto \delta_{\phi},
\end{split}
\end{equation}
	where \(\delta_{\phi}\) is determined by the following formula
	\begin{equation*}
		\phi(f)= f + \varepsilon \, \delta_{\phi}(f) \in \Gamma (U, \mc{O}_X + \varepsilon \, \mc{O}_X)= \Gamma(U, I_{k} \otimes \mc{O}_X)
	\end{equation*}
	for \(f \in \Gamma(U, \mc{O}_X)\), \(U\) open subset of \(X\), \(I_k=\frac{k[\varepsilon]}{\langle \varepsilon^2 \rangle}\) is the ring of dual numbers and $\phi \in \text{Lie}(\text{Aut}(X))$. Here, we have used the following identification (see \cite[Lemma 3.3]{MO}):
	\begin{equation*}
		\begin{split}
			\text{Lie}(\text{Aut}(X))=\{\phi:& \mc{O}_X \, \rar \, \text{I}_{k} \otimes \mc{O}_X		 \text{ is a homomorphism of } k \text{-algebras on } X ~|~ (r \otimes 1) \circ \phi= \text{id}\},
		\end{split}
	\end{equation*}
	where the map \(r : I_k \rightarrow k\) is given by sending \(\varepsilon\) to zero.
	
	\noindent
	Let $p : \mathcal{P} \rightarrow X$ be a principal \(H\)-bundle on a proper scheme \(X\). Consider the abstract group \(\text{Aut}^H(\mathcal{P})\) associated to $\mathcal{P}$ consisting of automorphisms of $\mathcal{P}$ preserving the action of \(H\). Since $ p : \mathcal{P} \rightarrow X$ is a geometric quotient, for any $\psi \in \text{Aut}^H(\mathcal{P})$ there is an automorphism of $X$, say $\theta(\psi)$ satisfying
	\begin{equation}\label{gammaeq}
		p \circ \psi =\theta(\psi) \circ p.
	\end{equation}
Thus we get a group homomorphism 
\begin{equation}\label{gamma}
	\theta : \text{Aut}^H(\mathcal{P}) \rightarrow \text{Aut}(X),
\end{equation}
which yields a map between group functors
\begin{equation*}
	\theta : \underline{Aut}^H(\mathcal{P}) \rightarrow \underline{Aut}(X).
\end{equation*}	
	The functor $\underline{Aut}^H(\mathcal{P})$ is represented by a group scheme, locally of finite type by \cite[Theorem 4.2]{Brauto}, which we will denote by $\text{Aut}^H(\mathcal{P})$. In particular, the neutral component \((Aut^H(\mathcal{P}))^{\circ}\) is an algebraic group. We have an isomorphism of Lie algebras (see \cite[Theorem 4.2]{Brauto}, \cite[Proposition 2.5, Section 4, Ch II]{D-M})
	\begin{equation}\label{P2}
		\begin{split}
			\beta_{\mathcal{P}}	:\text{Lie}(\text{Aut}^H(\mathcal{P})) & \stackrel{\cong} \longrightarrow	H^0(\mathcal{P}, \mathscr{T}_{\mathcal{P}})^H\\
			  \psi & \longmapsto \delta_{\psi}, 
		\end{split}
	\end{equation}
	where \(\delta_{\psi}\) is determined by the rule 
	\begin{equation*}
		\psi(\widetilde{f})= \widetilde{f} + \varepsilon \, \delta_{\psi}(\widetilde{f}) \in \Gamma (V, \mc{O_P} + \varepsilon \, \mc{O_P})= \Gamma(V, \text{I}_{k} \otimes \mc{O_P})
	\end{equation*}
	for \(\widetilde{f} \in \Gamma(V, \mc{O_P})^H\), \(V\) open subset of \(\mathcal{P}\) and $\psi \in \text{Lie}(\text{Aut}^H(\mathcal{P}))$. 
	
	The map $\theta$ in \eqref{gamma} induces the following map between the corresponding Lie algebras
	\begin{equation*}
		d \theta : \text{Lie}(\text{Aut}^H(\mathcal{P})) \longrightarrow \text{Lie}(\text{Aut}(X))
	\end{equation*}
	and the following diagram commutes
	\begin{equation}\label{preli1}
		\begin{tikzpicture}
			\matrix (m) [matrix of math nodes,row sep=3em,column sep=4em,minimum width=2em] {
				\text{Lie}(\text{Aut}^H(\mathcal{P})) & 	  H^0(\mathcal{P}, \mathscr{T}_{\mathcal{P}})^H\\
				\text{Lie}(\text{Aut}(X)) & H^0(X,\mathscr{T}_X) \\};
			\path[-stealth]
			(m-1-1) edge node [left] {$d \theta$} (m-2-1)
			edge  node [above] {$\beta_{\mathcal{P}}$} (m-1-2)
			(m-2-1) edge node [below] {$\alpha_X$} (m-2-2)
			(m-1-2) edge node [right] {$\eta(X)$} (m-2-2);
		\end{tikzpicture}	
	\end{equation}
	(see \cite[Figure 2]{DKP}).
	

	\section{Atiyah sequence for G-connection}

	Let \(X\) be a variety and \(H\) be a reductive liner algebraic group. Let \(p : \mathcal{P} \rightarrow X\) be a principal \(H\)-bundle. We recall the construction of the Atiyah sequence in this set-up (cf. \cite[Section 3]{DKP}). Since, the map \(p\) is smooth, we have the following relative cotangent sequence which is exact (see \cite[Theorem C.15]{Sernesi}).
\begin{equation*}
	0 \longrightarrow p^* \Omega_X \longrightarrow \Omega_{\mathcal{P}} \longrightarrow \Omega_{\mathcal{P} / X} \longrightarrow 0.
\end{equation*}
	Taking dual and using the fact \(\left(p* \Omega_X \right)^{\vee}=p^* \left( \Omega_X ^{\vee}\right) = p^* \mathscr{T}_X \) (see \cite[Proof of Proposition 1.8]{Stabref}), we get the following exact sequence
	
	\begin{equation}\label{S0.01}
		0 \longrightarrow \mathscr{T}_{\mathcal{P} / X}  \longrightarrow \mathscr{T}_{\mathcal{P}} \longrightarrow p^* \mathscr{T}_X \longrightarrow 0.
	\end{equation}
Since \(p\) is affine and \(H\) is reductive, taking invariant pushforward \(p^H_*:=p_*( \cdot )^H\), we get the following exact sequence 

\begin{equation}\label{preAt}
			0 \longrightarrow p^H_* \left( \mathscr{T}_{\mathcal{P} / X} \right)  \longrightarrow p^H_* (\mathscr{T}_{\mathcal{P}})
			 \longrightarrow p^H_* \left( p^* \mathscr{T}_X\right)  \longrightarrow 0,
\end{equation}
(see \cite[Section 2]{Nevins}).
\begin{defn}
	The Atiyah sheaf associated to the principal bundle is the subsheaf of \(H\)-invariants of $p_*\mathscr{T}_{\mathcal{P}}$:
	\begin{equation*}
		\begin{split}
			\mathcal{A}t(\mathcal{P}):=p_*^H\mathscr{T}_{\mathcal{P}} \subset p_*\mathscr{T}_{\mathcal{P}}.
		\end{split}
	\end{equation*}
\end{defn}
\begin{rmk}\rm {
		The Atiyah sheaf $\mathcal{A}t(\mathcal{P})$ is coherent as \(p: \mathcal{P} \rightarrow X \) is a principal \(H\)-bundle using \cite[Theorem 1.2]{Nevins}. Moreover, since \(p\) is surjective, using \cite[Proposition 8.4.5]{ega1}, we get that $\mathcal{A}t(\mathcal{P})$ is torsion-free.
		
	}
\end{rmk}

Note that \(p^H_* \left( p^* \mathscr{T}_X\right) \cong \mathscr{T}_X\) (see \cite[Section 2]{Nevins}). Also, using \cite[Proposition 3.6]{DKP}, we have \(p^H_* \left( \mathscr{T}_{\mathcal{P} / X} \right)\) is isomorphic to the adjoint bundle \(\text{ad}(\mathcal{P}):=\mathcal{P} \times_H \mathfrak{h}\), where \(H\) acts on its Lie algebra $\mathfrak{h}$ via the adjoint representation. Hence, the exact sequence \eqref{preAt} gives the following short exact sequence 
\begin{equation}\label{Atseq}
		0 \longrightarrow \text{ad}(\mathcal{P})  \stackrel{\iota_0}\longrightarrow \mathcal{A}t(\mathcal{P}) 	\stackrel{\eta}\longrightarrow  \mathscr{T}_X  \longrightarrow 0.
\end{equation}
called the Atiyah sequence. Note that the maps in the Atiyah sequence preserves the natural Lie algebra structures.

\subsection{Atiyah sequence for group action} \label{Atiyah sequence for group action}

Let \(G\) be a connected  algebraic group (not necessarily linear algebraic group) acting on a complete variety \(X\) via the regular morphism $\rho : G \times X \rightarrow X$.  Now $\rho$ induces the following homomorphism of algebraic groups (see \cite[proof of Corollary 2.2]{Brauto}, \cite[Remark 2.4]{Br18})
\begin{equation*}
	\bar{\rho} : G \longrightarrow \text{Aut}^0(X), \, g \mapsto (\rho_g : x \mapsto\rho(g, x)  ).
\end{equation*}
This further induces the following map between the corresponding Lie algebras
\begin{equation*}
	\begin{split}
		d \bar{\rho} ~:~ \mathfrak{g} & \longrightarrow \text{Lie}(\text{Aut}(X)), \\
		(\varphi : \text{Spec}(I_k) \rightarrow G) &\longmapsto (\bar{\rho} \circ \varphi : \text{Spec}(I_k) \rightarrow \text{Aut}^0(X)) . 
	\end{split}
\end{equation*}
Composing it with the isomorphism \eqref{P1}, we get the following Lie algebra homomorphism (see \cite[Proposition 4.4, Section 4, Ch II]{D-M}, \cite[Lemma 1.2]{AGS} and \cite{LI10})
\begin{equation*}
	\alpha= \alpha_X \circ d \bar{\rho} ~:~ \mathfrak{g} \longrightarrow H^0(X, \mathscr{T}_X).
\end{equation*}
 This induces the following evaluation map between sheaves 
\begin{equation}\label{funda_ac}
	\begin{split}
		\zeta ~:~ & \mathcal{O}_X \otimes_{k}  \mathfrak{g} \longrightarrow \mathscr{T}_X \text{ given by},\\
		& \mathcal{O}_X(U) \otimes_{k}  \mathfrak{g} \longrightarrow \mathscr{T}_X(U), \, f \otimes \delta \mapsto f \, \alpha(\delta)|_{U}, 
	\end{split}
\end{equation}
where \(U\) is an open subset of \(X\), \(f \in \mathcal{O}_X(U)\) and \(\delta \in \mathfrak{g}\). Note that $\zeta$ preserves the natural Lie algebra structures between the sheaves, $\mathcal{O}_X \otimes \mathfrak{g}$ has a natural Lie bracket induced from the Lie bracket of $\mathfrak{g}$.  

\begin{rmk}\label{DG-Lieq}{\rm 
	Note that for any point \(x \in X\), the map $\zeta$ yields the following map between the the fibers
	\begin{equation*}
		\zeta_x : \mathfrak{g} \longrightarrow \mathscr{T}_x(X): = \mathscr{T}_X \otimes_{\mathcal{O}_X} \, k(x), 
	\end{equation*}
	which is same as $\alpha$. Also, $\zeta_x$ is identified with the differential of the orbit map \(G \rightarrow X, \, g \mapsto \rho(g, x)\) (cf. \cite[the map $op_X$ on Page 2]{Br13}). The action map $\rho : G \times X \rightarrow X$ induces the following map between the corresponding tangent sheaves (see \eqref{S0.01})
	\begin{equation*}
		 \mathscr{T}_{G \times X} \stackrel{d \rho} \longrightarrow  \rho^* \mathscr{T}_X .
	\end{equation*}
Let \(p_1 : G \times X \rightarrow G\) and \(p_2 : G \times X \rightarrow X\) denote the corresponding projection maps, respectively. Then we have \(\mathscr{T}_{G \times X} \cong p_1^* \, \mathscr{T}_G \oplus p_2^* \, \mathscr{T}_X\). Let \(j : p_1^* \, \mathscr{T}_G \hookrightarrow  \mathscr{T}_{G \times X} \) denote the inclusion map. Consider the following composition of maps:
		\begin{equation}\label{Eq-1}
	p_1^* \, \mathscr{T}_G \stackrel{j} \longrightarrow \mathscr{T}_{G \times X} \stackrel{d \rho} \longrightarrow  \rho^* \mathscr{T}_X.
	\end{equation}
	Let $\nu : X \cong \{1_G\} \times X \hookrightarrow G \times X$ denote the inclusion map given by \(x \mapsto (1_G, x)\). Then pulling back the sequence \eqref{Eq-1} via the map $\nu$ and using the fact that \(\rho \circ \nu =\text{Id}_X\) and \( p_1^* \, \mathscr{T}_G \cong \mathcal{O}_{G \times X} \otimes_{k} \, \mathfrak{g}\) (see \cite[p. 13 ]{BrStr}), we get the following map
	\begin{equation*}
		\zeta' : \mathcal{O}_X \otimes_{k } \mathfrak{g} \longrightarrow \mathscr{T}_X.
	\end{equation*}
	Note that, on fibers, the map $\zeta'$ is given by the differential of the orbit map \(G \rightarrow X, \, g \mapsto \rho(g, x)\). Hence, using Nakayama's lemma we have $\zeta=\zeta'$.	
	}

\end{rmk}

\noindent
Denote by $\mathcal{V}:=\mathcal{O}_X \otimes_{k} \mathfrak{g}$ and consider the map
\begin{equation*}
\eta-\zeta ~ : ~ \mathcal{A}t(\mathcal{P}) \oplus \mathcal{V} \rightarrow \mathscr{T}_X, \, (v, w) \rightarrow \eta(v)-\zeta(w).
\end{equation*}

\begin{defn}
	Define the \(G\)-Atiyah sheaf 
\begin{equation*}
		\mathcal{A}t_{\rho}(\mathcal{P}):= \text{Ker}(\eta - \zeta) 
\end{equation*}
	as a subsheaf of \(\mathcal{A}t(\mathcal{P}) \oplus \mathcal{V}\).
\end{defn}
Clearly, \(\mathcal{A}t_{\rho}(\mathcal{P})\) is torsion-free. Note that being kernel, \(\mathcal{A}t_{\rho}(\mathcal{P})\) acquires a natural Lie algebra structure from the component wise Lie algebra structure from \(\mathcal{A}t(\mathcal{P}) \oplus \mathcal{V}\).  We have the projection map 
\[\text{pr}_2 : \mathcal{A}t_{\rho}(\mathcal{P}) \rightarrow \mathcal{V}\]
and the map
\[\iota : \text{ad}(\mathcal{P}) \rightarrow \mathcal{A}t_{\rho}(\mathcal{P}), \, \iota(v)=(\iota_0(v), 0).\]
Note that $\iota$ is injective and $\text{Ker}(\text{pr}_2)=\text{Im}(\iota)$; hence, we have the following exact sequence
\begin{equation}\label{G-Atseq}
	0 \longrightarrow \text{ad}(\mathcal{P}) \stackrel{\iota} \longrightarrow  \mathcal{A}t_{\rho}(\mathcal{P}) \stackrel{pr_2} \longrightarrow \mathcal{V} \longrightarrow 0.
\end{equation}
We call this the \(G\)-Atiyah sequence for the principal bundle $\mathcal{P}$. Observe that the above sequence is obtained from the Atiyah sequence \eqref{Atseq} by pulling back via the map $\zeta$, defined in \eqref{funda_ac} (see \cite[Chapter III, Section 1]{HShomAG}). Since both the sheaves \(\text{ad}(\mathcal{P})\) and $\mathcal{V}$ are locally free, \(\mathcal{A}t_{\rho}(\mathcal{P})\) is also locally free. Clearly, \(\text{pr}_2\) preserves the Lie algebra structure. Note that $\iota$ also preserves the Lie algebra structure as $\iota_0$ does so. Hence, the maps in the above sequence \eqref{G-Atseq} preserves the natural Lie algebra structures between the sheaves.

\begin{defn}
	A \(G\)-connection on the principal \(H\)-bundle $\mathcal{P}$ is a $\mathcal{O}_X$-linear splitting \(\), i.e., a $\mathcal{O}_X$-module map
	\[\lambda: \mathcal{V} \rightarrow \mathcal{A}t_{\rho}(\mathcal{P}) \text{ such that }\text{pr}_2 \circ \lambda=Id_{\mathcal{V}}.\]
	A \(G\)-connection \(\lambda\) is said to be flat or integrable if moreover \(\lambda\) preserves the Lie algebra structures of the sheaves.
\end{defn}

\begin{rmk}\label{Gat_class}{\rm 
	The \(G\)-Atiyah sequence \eqref{G-Atseq} defines  a class \(a_{\rho}(\mathcal{P}) \in \text{Ext}^1(\mathcal{V}, \text{ad}(\mathcal{P}))\). Thus the principal \(H\)-bundle admits a \(G\)-connection if and only if \(a_{\rho}(\mathcal{P})=0\). Note that
	\[\text{Ext}^1(\mathcal{V}, \text{ad}(\mathcal{P})) =H^1(X, \Hom(\mathcal{V}, \text{ad}(\mathcal{P})))=H^1(X, \text{ad}(\mathcal{P}) \otimes_{k}  \mathfrak{g}^{\vee} ) . \] 
Hence, \(\text{Ext}^1(\mathcal{V}, \text{ad}(\mathcal{P}))\) is isomorphic to the direct sum of finitely many copies of \(H^1(X, \text{ad}(\mathcal{P}))\).  In particular, if \(H^1(X, \text{ad}(\mathcal{P}))=0, \) then $\mathcal{P}$ admits a \(G\)-connection.
}
	
\end{rmk}

\subsection{\(G\)-connection on vector bundle}\label{Gconnection on vector bundle}

 As before, let \(G\) be a connected  algebraic group acting on a complete variety \(X\) via the regular morphism $\rho : G \times X \rightarrow X$. Let $\pi: \mathcal{E} \rightarrow X$ be a vector bundle of rank \(r\) over \(X\) and $p: \mathcal{P} \rightarrow X$ be the associated \(GL(r, k)\)-bundle, called the frame bundle. Then by Remark \ref{Gat_class}, the G-Atiyah sequence \eqref{G-Atseq} for the frame bundle defines  a class \(a_{\rho}(\mathcal{P}) \in H^1(X, \Hom(\mathcal{V}, \text{ad}(\mathcal{P})))\). Consider the following canonical isomorphisms.
\begin{equation*}
	\begin{split}
	&\mathcal{E} ^{\vee} \otimes \mathcal E \cong ad(\mathcal{P}) ~ \text{(cf. \cite[Proposition 9]{At57})},\\
	&	\Hom(\mathcal{V}, \text{ad}(\mathcal{P}))) \cong \Hom(\mathcal{V}, \mathcal{E} ^{\vee} \otimes \mathcal E) \cong \mathcal{V}^{\vee} \otimes \mathcal{E} ^{\vee} \otimes \mathcal E  \cong \Hom(\mathcal{E}, \mathcal{E} \otimes \mathcal{V}^{\vee}).	
	\end{split}
\end{equation*}
Let us denote by \(b_{\rho}(\mathcal{E}) \in H^1(X, \Hom(\mathcal{E}, \mathcal{E} \otimes \mathcal{V}^{\vee}))\) the image of the class \(a_{\rho}(\mathcal{P})\). Then the class  \(b_{\rho}(\mathcal{E})\) corresponds to the following extension (see \cite[Proposition 2]{At57}, \cite[Section 3]{Gray61})
\begin{equation}\label{GAtfVB}
		0 \longrightarrow \mathcal{E} \otimes \mathcal{V}^{\vee}  \stackrel{} \longrightarrow  D_{\rho}(\mathcal{E}) \stackrel{} \longrightarrow \mathcal E \longrightarrow 0.
\end{equation}
Note that the class \(a_{\rho}(\mathcal{P})\) is zero if and only if the class \(b_{\rho}(\mathcal{E})\) is zero. Thus we can define G-connection for the vector bundle $\mathcal{E}$ as an $\mathcal{O}_X$-linear splitting of the sequence \eqref{GAtfVB}.

\begin{prop}
	Let $\mathcal{E}$ be a vector bundle on a complete variety \(X\) admitting a \(G\)-connection. Then there is a \(k\)-linear sheaf homomorphism 
\begin{equation*}
	\nabla : \mathcal{E} \longrightarrow \mathcal{E} \otimes_{\mathcal{O}_X} \mathcal{V}^{\vee}
\end{equation*}
satisfying the following Leibniz like condition
\begin{equation*}
	\nabla(f s)=f \, \nabla (s) + s \otimes \zeta^{\vee}(df),
\end{equation*}
where \(f \in \mathcal{O}_X(U)\) and \(s \in \mathcal{E}(U)\) for open subset \(U \subset X\). Here $\zeta^{\vee}$ is the dual of the map $\zeta$ defined in \eqref{funda_ac}. Note that \(d : \mathcal{O}_X \rightarrow \Omega^1_X\) is the universal derivation and we have the natural map \(\Omega^1_X \rightarrow (\Omega^1_X)^{\vee \vee}\).
\end{prop}

\begin{proof} Let $\mathcal{P}$ denote the principal frame bundle associated to $\mathcal{E}$.
	Consider the following canonical isomorphism of $\mathcal{O}_{\mathcal{P}}$-modules (see \cite[proof of Proposition 4.6]{DKP})
	\begin{equation*}
		\Psi : p^* \mathcal{E} \longrightarrow \mathcal{O}_{\mathcal{P}}^{\oplus r}, \, \tilde{s} \longmapsto (\tilde{f}_1, \ldots, \tilde{f}_r),
	\end{equation*}
	where $\tilde{s}$ is a section of \(p^* \mathcal{E}\) over the open subset \(V \subset \mathcal{P}\) and 
	\begin{equation*}
		\tilde{s}(e)=(e, \tilde{f}_1 e_1 + \ldots + \tilde{f}_r e_r)
	\end{equation*}
	for all \(e \in V\) and \(e=(e_1, \ldots, e_r)\) is an ordered basis of the fiber $\mathcal{E}(x)$, \(x=p(e)\). Note that $\Psi$ is  \(GL(r, k)\)-equivariant. Define the map
	\begin{equation*}
		\widetilde{\nabla}: \mathcal{E} \otimes_k \mathcal{V} \longrightarrow \mathcal{E} 
	\end{equation*}
	as follows: Let $\lambda : \mathcal{V} \rightarrow \mathcal{A}t_{\rho}(\mathcal{P})$ be a \(G\)-connection on $\mathcal{P}$. Let \(s\) and \(w\) be local sections of $\mathcal{E}$ and $\mathcal{V}$, respectively, over  an affine open subset \(U=\text{Spec}(B) \subset X \). Let \(\lambda(w)=(\widetilde{\delta}_w, w)\), where \(\widetilde{\delta}_w\) is a local section of \(\mathcal{A}t(\mathcal{P})\) such that \(\eta(\widetilde{\delta}_w)=\zeta(w)\). Let \(p^{-1}(U)=\text{Spec}(A) \subset \mathcal{P}\). Then \((p^{*} \mathcal{E})(p^{-1}(U))=\mathcal{E}(U) \otimes_B A \). So \(s \otimes 1\) defines a local section of $p^{*} \mathcal{E}$  over \(p^{-1}(U)\), where \(1\) denotes the multiplicative identity of the ring \(A\). Write
	\begin{equation*}
		\Psi(s \otimes 1)=(\tilde{f}_1, \ldots, \tilde{f}_r) \in \mathcal{O}_{\mathcal{P}}^{\oplus r}.
	\end{equation*}
Then, define
	\begin{equation*}
		\widetilde{\nabla}(s \otimes w)=\Psi^{-1}(\widetilde{\delta}_w(\tilde{f}_1), \ldots, \widetilde{\delta}_w(\tilde{f}_r)).
	\end{equation*}
	Note that \(\widetilde{\nabla}(s \otimes w)\) is an element of \(\mathcal{E}(U)\) as $\Psi$ is  \(GL(r, k)\)-equivariant. Since, the map $\Psi$ is $\mathcal{O}_{\mathcal{P}}$-linear and $\lambda$ is $\mathcal{O}_X$-linear, we have
	\begin{equation}\label{cU-linear}
		\widetilde{\nabla}(cs \otimes w)=c \, \widetilde{\nabla}(s \otimes w)  \text{ and }  \widetilde{\nabla}(s \otimes f w)=f \, \widetilde{\nabla}(s \otimes w).
	\end{equation}
	Also, note that \[fs \otimes 1=(p^{\sharp}f) (s \otimes 1).\] 
	Then, $$\Psi(fs \otimes 1)=(p^{\sharp}f) \, \Psi(s \otimes 1)=((p^{\sharp}f) \tilde{f}_1, \ldots, (p^{\sharp}f) \tilde{f}_r).$$ Thus, we have
	\begin{equation}\label{leib0}
		\begin{split}
			\widetilde{\nabla}(fs \otimes w) & =\Psi^{-1}(\widetilde{\delta}_w((p^{\sharp}f) \tilde{f}_1), \ldots, \widetilde{\delta}_w( (p^{\sharp}f) \tilde{f}_r)).
		\end{split}
	\end{equation}
	Note that for \(i=1, \ldots, r\), we have
	\begin{equation}\label{leib2}
		\widetilde{\delta}_w((p^{\sharp}f ) \tilde{f}_i)=(p^{\sharp}f ) \, \widetilde{\delta}_w(\tilde{f}_i) + \tilde{f}_i \, \widetilde{\delta}_w((p^{\sharp}f )).
	\end{equation}
	
Then from \eqref{leib0} and \eqref{leib2}, we have
\begin{equation}\label{preleib}
	\begin{split}
		\widetilde{\nabla}(fs \otimes w)=& (p^{\sharp}f) \Psi^{-1}(\widetilde{\delta}_w(\tilde{f}_1), \ldots, \widetilde{\delta}_w(\tilde{f}_r)) + \widetilde{\delta}_w((p^{\sharp}f )) \Psi^{-1}(\tilde{f}_1, \ldots, \tilde{f}_r)\\
		= & f \, \widetilde{\nabla}(s \otimes w) + \zeta(w)(f) \, s,
	\end{split}
	\end{equation}	
	where we have used
	\begin{equation*}
		\eta(\widetilde{\delta}_w)(f)=p_*(\widetilde{\delta}_w(p^{\sharp}(f) )).
	\end{equation*}
	Let us define
	\begin{equation*}
		\nabla : \mathcal{E} \longrightarrow \mathcal{E} \otimes_{\mathcal{O}_X} \mathcal{V}^{\vee}
	\end{equation*}
	as follows: for a section \(s\) of $\mathcal{E}$ over an open subset \(U \subset X \) consider the $\mathcal{O}_X(U)$-linear map
	\begin{equation*}
		\widetilde{\nabla}^U_s : \mathcal{V}(U) \longrightarrow \mathcal{E}(U), \text{ given by } w \longmapsto \widetilde{\nabla}^U(s \otimes w).
	\end{equation*}
	Finally define
	\begin{equation*}
		\nabla^U(s)=\Upsilon \left( \widetilde{\nabla}^U_s \right), 
	\end{equation*}
	where $\Upsilon : \Hom(\mathcal{V}, \mathcal{E}) \cong \mathcal{E} \otimes_{\mathcal{O}_X} \mathcal{V}^{\vee}$ is the canonical isomorphism. Then from \eqref{cU-linear} and \eqref{preleib}, it follows that $\nabla$ is a \(k\)-linear map satisfying Leibniz rule.
	\end{proof}

\begin{rmk}{\rm 
	The above proposition suggests that we can define \(G\)-connection for any coherent $\mathcal{O}_X$-module as well. Note that when \(X\) is nonsingular and the map $\zeta$ is surjective any coherent $\mathcal{O}_X$-module admitting \(G\)-connection is necessarily locally free (using \cite[Proposition 1.2, p. 211]{DBorel}). This does not hold for singular varieties; for example, the cotangent sheaf always admits a \(G\)-connection (use Proposition \ref{eqtoconn}) although it need not be locally free.
}
\end{rmk}

\subsection{Homomorphisms and induced \(G\)-connections}

Let \(G\) be a connected  algebraic group acting on a complete variety \(X\) via the regular morphism $\rho : G \times X \rightarrow X$. Let \(\tau : G_1 \rightarrow G\) be a homomorphism of connected algebraic groups. This induces an action of \(G_1\) on \(X\) as follows:
\begin{equation*}
	\rho_1 : G_1 \times X \longrightarrow X, \text{ given by } (g_1, x) \longmapsto \rho(\tau(g_1), x).
\end{equation*}
Now $\rho_1$ induces a homomorphism of algebraic groups 
\begin{equation*}
	\bar{\rho_1} : G_1 \longrightarrow \text{Aut}^0(X), \, g_1 \mapsto ((\rho_1)_{g_1} : x \longmapsto\rho(\tau(g_1), x) ).
\end{equation*}
Thus we have \(\bar{\rho_1}= \bar{\rho} \circ \tau\). This implies that the corresponding maps between their Lie algebras satisfy 
\(d\bar{\rho_1}= d\bar{\rho} \circ d\tau\). This implies we get \(\alpha_1= \alpha \circ d\tau\).  Hence, the corresponding map defined in \eqref{funda_ac} is 
\begin{equation*}
	\begin{split}
		\zeta_1 ~:~ & \mathcal{O}_X \otimes_{k}  \mathfrak{g}_1 \longrightarrow \mathscr{T}_X \text{ given by},\\
		& \mathcal{O}_X(U) \otimes_{k}  \mathfrak{g}_1 \longrightarrow \mathscr{T}_X(U), \, f \otimes \delta_1 \mapsto f \, \alpha_1(\delta)|_{U},
	\end{split}
\end{equation*}
where \(U\) is an open subset of \(X\), \(f \in \mathcal{O}_X(U)\) and \(\delta \in \mathfrak{g}_1\). Thus, we get 
\begin{equation*}
	\zeta_1= \zeta \circ d\tau. 
\end{equation*}
Hence, the following diagram commutes
\begin{equation*}
	\begin{tikzpicture}[description/.style={fill=white,inner sep=2pt}]
		\matrix (m) [matrix of math nodes, row sep=3em,
		column sep=2.5em, text height=1.5ex, text depth=0.25ex]
		{ \mathcal{A}t(\mathcal{P}) \oplus (\mathcal{O}_X \otimes_{k} \mathfrak{g}_1)    & &   \mathcal{A}t(\mathcal{P}) \oplus (\mathcal{O}_X \otimes_{k} \mathfrak{g}_)  \\
			\mathscr{T}_X    & &   \mathscr{T}_X  \\ };
		\path[->]  (m-1-1) edge node[auto] {}(m-1-3);
		\path[ ->] (m-2-1) edge node[below] {}(m-2-3);
		\path[->] (m-1-1) edge node[auto] {$\eta - \zeta_1$}(m-2-1);
		\path[->] (m-1-1) edge node[above] {$Id \oplus d\tau$}(m-1-3);
		\path[ ->] (m-1-3) edge node[auto] {$\eta - \zeta$} (m-2-3);
		\path[->] (m-2-1) edge node[above] {$Id$} (m-2-3);
	\end{tikzpicture}	.
\end{equation*}
This further induces the following map between the kernels,
\begin{equation*}
\widetilde{\tau}:	\mathcal{A}t_{\rho_1}(\mathcal{P}) \longrightarrow \mathcal{A}t_{\rho}(\mathcal{P}), \, (v, w_1) \longmapsto (v, d\tau(w_1)).
\end{equation*}
Thus, we get the following commutative diagram of the corresponding Atiyah sequences.

\begin{equation}\label{homotoindeq1}
	\xymatrixrowsep{1.8pc} \xymatrixcolsep{2.2pc}
	\xymatrix{
		0\ar@{->}[r]&
		\text{ad}(\mathcal{P})\ar@{->}[r]^{\iota_1}\ar@{->}[d]^{Id}& \mathcal{A}t_{\rho_1}(\mathcal{P})\ar@{->}[d]^{\widetilde{\tau}}\ar@{->}[r]^{pr_2}& \mathcal{O}_X \otimes_{k } \mathfrak{g}_1\ar@{->}[d]_{d\tau}\ar@{->}[d]\ar@{->}[r]&0\\
		0\ar@{->}[r]&
		\text{ad}(\mathcal{P})\ar@{->}[r]^{\iota}&\mathcal{A}t_{\rho}(\mathcal{P})\ar@{->}[r]^{pr_2}&\mathcal{O}_X \otimes_{k } \mathfrak{g}\ar@{->}[r]&{0}.
	}
\end{equation}
\begin{prop}\label{H2H1}
	A \(G\)-connection on the principal \(H\)-bundle $\mathcal{P}$ induces a \(G_1\)-connection on $\mathcal{P}$.
\end{prop}

\begin{proof}
	If the principal \(H\)-bundle $\mathcal{P}$ admits a \(G\)-connection, the map $\iota$ in \eqref{homotoindeq1} admits a splitting. This will produce a splitting of $\iota_1$ by the commutativity of the exact sequences in \eqref{homotoindeq1}. Thus $\mathcal{P}$ will admit a \(G_1\)-connection as well. 
\end{proof}

Next we consider \(G\)-connections induced on associated bundles of a principal bundle. This has been studied for logarithmic connections in \cite[Section 5]{DKP}. Let \(X\) be a complete variety and $\tau: H \rightarrow H_1$ be a homomorphism between reductive algebraic groups \(H\) and \(H_1\). Let $p : \mathcal{P} \rightarrow X$ be a  principal \(H\)-bundle. The associated principal \(H_1\)-bundle \(\mathcal{P}_1 := \mathcal{P} \times^H  H_1\) is constructed as the quotient \[\left( \mathcal{P} \times  H_1 \right)/ \sim \] where 
\[(e,h_1) \sim (eh,\tau(h)^{-1}h_1) \text{ for all } h \in H.\] 
The projection map \[p': \mathcal{P}_1  \rightarrow X \text{ is given by }  p'([e,h_1])=p(e).\] 

\begin{prop}\label{inducedconn3.1}
	With the notation as above, a (flat) \(G\)-connection on the principal \(H\)-bundle $\mathcal{P}$ induces a (flat)  \(G\)-connection on the associated principal $H_1$-bundle $\mathcal{P}_1$.
\end{prop}

\begin{proof}
	Note that we have the following commutative diagram of the corresponding Atiyah sequences
	
	\begin{equation*}
		\xymatrixrowsep{1.8pc} \xymatrixcolsep{2.2pc}
		\xymatrix{
			0\ar@{->}[r]&
			\text{ad}(\mathcal{P})\ar@{->}[r]^{\iota_0}\ar@{->}[d]^{\text{ad}(\tau)}& \mathcal{A}t(\mathcal{P})\ar@{->}[d]^{\text{At}(\tau)}\ar@{->}[r]^{\eta}&\mathscr{T}_X\ar@{->}[d]_{Id}\ar@{->}[d]\ar@{->}[r]&0\\
			0\ar@{->}[r]&
			\text{ad}(\mathcal{P}_1)\ar@{->}[r]^{\iota_0'}&\mathcal{A}t(\mathcal{P}_1)\ar@{->}[r]^{\eta'}&\mathscr{T}_X\ar@{->}[r]&{0},
		}
	\end{equation*}
(see \cite[proof of the Proposition 5.1]{DKP}). From the commutativity of the second square in the above diagram, we get the following commutative diagram
\begin{equation*}
	\begin{tikzpicture}[description/.style={fill=white,inner sep=2pt}]
		\matrix (m) [matrix of math nodes, row sep=3em,
		column sep=2.5em, text height=1.5ex, text depth=0.25ex]
		{ \mathcal{A}t(\mathcal{P}) \oplus \mathcal{V}    & &   \mathscr{T}_X  \\
			\mathcal{A}t(\mathcal{P}_1) \oplus \mathcal{V}     & &   \mathscr{T}_X  \\ };
		\path[->]  (m-1-1) edge node[auto] {}(m-1-3);
		\path[ ->] (m-2-1) edge node[below] {}(m-2-3);
		\path[->] (m-1-1) edge node[auto] {$\text{At}(\tau) \oplus Id$}(m-2-1);
		\path[->] (m-1-1) edge node[above] {$\eta-\zeta$}(m-1-3);
		\path[ ->] (m-1-3) edge node[auto] {$Id$} (m-2-3);
		\path[->] (m-2-1) edge node[above] {$\eta'-\zeta$} (m-2-3);
	\end{tikzpicture}	.
\end{equation*} 
This further induces the following map between the kernels,
\begin{equation*}
	\text{At}_{\rho}(\tau):	\mathcal{A}t_{\rho}(\mathcal{P}) \longrightarrow \mathcal{A}t_{\rho}(\mathcal{P}_1), \, (v, w) \longmapsto (\text{At}(\tau)(v), w).
\end{equation*}
Thus, we get the following commutative diagram of the corresponding Atiyah sequences.

\begin{equation}\label{homotoindeq1.1}
	\xymatrixrowsep{1.8pc} \xymatrixcolsep{2.2pc}
	\xymatrix{
		0\ar@{->}[r]&
		\text{ad}(\mathcal{P})\ar@{->}[r]^{\iota}\ar@{->}[d]^{\text{ad}(\tau)}& \mathcal{A}t_{\rho}(\mathcal{P})\ar@{->}[d]^{\text{At}_{\rho}(\tau)}\ar@{->}[r]^{pr_2}& \mathcal{V}\ar@{->}[d]_{Id}\ar@{->}[d]\ar@{->}[r]&0\\
		0\ar@{->}[r]&
		\text{ad}(\mathcal{P}_1)\ar@{->}[r]^{\iota_1}&\mathcal{A}t_{\rho}(\mathcal{P}_1)\ar@{->}[r]^{pr_2}&\mathcal{V}\ar@{->}[r]&{0}.
	}
\end{equation}
   Let the principal $H$-bundle $\mathcal{P}$ admits a \(G\)-connection $h$, then from the commutativity of \eqref{homotoindeq1.1}, $\text{At}_{\rho}(\tau) \circ h$ will be a \(G\)-connection on $\mathcal{P}_1$. Note that the map \(\text{At}(\tau)\) preserves the Lie algebra structure; hence, does the map  $\text{At}_{\rho}(\tau)$. Thus, if $h$ is  flat, the induced \(G\)-connection $\text{At}_{\rho}(\tau) \circ h$ is also flat.
\end{proof}

\begin{prop}\label{H1_2H}
	Let $\tau : H \rightarrow H_1$ be an injective homomorphism between reductive algebraic groups \(H\) and \(H_1\). Let $p : \mathcal{P} \rightarrow X$ be a  principal \(H\)-bundle. Then a (flat) \(G\)-connection on the associated principal \(H_1\)-bundle $\mathcal{P}_1:= \mathcal{P} \times^H  H_1$ induces a (flat) \(G\)-connection on the principal $H$-bundle $\mathcal{P}$.
\end{prop} 

\begin{proof}
	 We have a map \(\bar \gamma : \text{ad}(\mathcal{P}_1) \rightarrow   \text{ad}(\mathcal{P})\) such that \(\bar \gamma \circ \text{ad}(\tau)= Id_{\text{ad}(\mathcal{P})}\) (see \cite[proof of the Proposition 5.3]{DKP},). Note that a  \(G\)-connection $\lambda_1$ on $\mathcal{P}_1$ induces a splitting 
	\[\lambda_1' : \mathcal{A}t(\mathcal{P}_1) \rightarrow \text{ad}(\mathcal{P}_1) \] 
	of $\iota_1$ in \eqref{homotoindeq1.1}. Then \(\bar \gamma \circ \lambda_1' \circ \text{At}_{\rho}(\tau)\) gives a splitting of $\iota$. Hence, we get a \(G\)-connection on \(\mathcal{P}\). Note that the map \(\bar \gamma\) preserves the Lie algebra structure. If \(G\)-connection $h_1$ on $\mathcal{P}_1$ is flat, we can choose the splitting \(\lambda_1'\) to preserve the Lie algebra structure. Hence, \(\bar \gamma \circ \lambda_1' \circ \text{At}_{\rho}(\tau)\) will give a flat \(G\)-connection on \(\mathcal{P}\).
\end{proof}

As a corollary, we get the following result analogous to a result for holomorphic connection in \cite[Proposition 2.2]{AB03}.

\begin{cor}\label{exiat_ad}
Let \(H\) be a semisimple linear algebraic group and $\mathcal{P}$ be a principal \(H\)-bundle on \(X\). Then $\mathcal{P}$ admits a (flat) \(G\)-connection if and only if  the associated adjoint bundle \(\text{ad}(\mathcal{P})\) admits a (flat) \(G\)-connection.
\end{cor}

We now recall the notion of Levi subgroup. Let \(H\) be a connected reductive affine algebraic group over \(k\). Let \(P\) be  a parabolic subgroup of \(H\) and \(R_u(P)\) be the unipotent radical of \(P\). There is a connected closed subgroup \(L \subset P\), called the Levi subgroup, such the restriction to \(L\) of the quotient map
\begin{equation*}
	\pi : P \longrightarrow P / R_u(P)
\end{equation*}
is an isomorphism. Note that \(L\) is also reductive (see \cite[Section 30.2]{Hump}, \cite[Page 559]{AGmil}). Given a principal \(H\)-bundle $\mathcal{P}$ on a projective variety \(X\) over \(k\) there is a natural reduction of structure group $\mathcal{P}_L \subset \mathcal{P}$ to a certain Levi subgroup \(L \subset H\) associated to $\mathcal{P}$ (see \cite{BBN05}). As a consequence of Proposition \ref{H2H1} and Proposition \ref{H1_2H}, we get the following.

\begin{cor}\label{GconnLeviRed}
When \(X\) is projective, with the notation as above, to study existence of (flat) \(G\)-connection on $\mathcal{P}$ it suffices to check existence of (flat) \(G\)-connection on a Levi reduction.
\end{cor}

We now extend Corollary \ref{exiat_ad} for any reductive group \(H\) (cf. \cite[Section 3]{AB03}). Let \(Z^0(H) \subset H\) denote the connected component of the center of \(H\) containing the identity element. Let \(H':=H /Z^0(H) \) be the quotient. Note that \(H'\) is semisimple. Let \([H \, , \, H]\) denote the commutator and \(A:=H / [H \, , \, H]\) be the abelian quotient. For a principal \(H\)-bundle $\mathcal{P}$, let $\mathcal{P}_1$ (respectively, $\mathcal{P}_2$) denote the principal \(H'\)-bundle (respectively, principal \(A\)-bundle) obtained by extending the structure group of $\mathcal{P}$ under the obvious projection map of \(H\) to \(H'\) (respectively, \(A\)). Then we have the following (cf. \cite[Lemma 3.4]{AB02}).

\begin{prop}\label{exist_proj}
	The principal \(H\)-bundle $\mathcal{P}$ admits a \(G\)-connection if and only if both the bundles $\mathcal{P}_1$ and $\mathcal{P}_2$ admit \(G\)-connection.
\end{prop}

\begin{proof}
	The forward direction is a consequence of Proposition \ref{H2H1}. For the converse direction, consider the diagonal homomorphism $\Delta : H \rightarrow H' \times A$ induced by the projections from \(H\) to \(H'\) and \(A\). Note that $\Delta$ is a surjective map between linear algebraic group with \(\text{ker}(\Delta)=Z^0(H) \cap [H \, , \, H]\), which is finite (see \cite[Lemma 19.5]{Hump}). Hence, $\Delta$ is an \'etale morphism. Let \(\mathcal{P}(\Delta):= \mathcal{P} \times^H (H' \times A) \) denote the principal \((H' \times A)\)-bundle obtained by extending the structure group of $\mathcal{P}$ by the homomorphism $\Delta$. Now consider the map
	\begin{equation*}
		 \mathcal{P} \longrightarrow \mathcal{P}(\Delta), \text{ given by } z \longmapsto [z,(1,1)].
	\end{equation*}
This is an \'etale morphism as \'etale locally, the above map is of the form \(Id \times \Delta\). Note that there is an isomorphism of principal bundles
	\begin{equation*}
\begin{split}
			\mathcal{P}(\Delta) & \stackrel{\cong}\longrightarrow \mathcal{P}_1 \times_X \mathcal{P}_2, \text{ given by }\\
			 [z, (h',a)] & \longmapsto [(z, h'), (z, a)].
\end{split}
	\end{equation*}
The inverse map is given by
\begin{equation*}
	[(z, h'),(z h, a)] \longmapsto [z,(h',  [h]^{-1} a)],
\end{equation*}
where \([h]\) denote the image of \(h\) in \(A\). Hence, the composite map \(\widetilde{\Delta} : \mathcal{P} \rightarrow \mathcal{P}_1 \times_X \mathcal{P}_2\) is an \'etale morphism. As a consequence, the Atiyah sequence for $\mathcal{P}$ and  \(\mathcal{P}':=\mathcal{P}_1 \times_X \mathcal{P}_2\) coincide (see \cite[Lemma 1.12]{BMW12}). Note that \'etaleness of the map \(\widetilde{\Delta} \) implies that $d\widetilde{\Delta} : \mathscr{T}_{\mathcal{P}} \stackrel{\cong} \rightarrow \widetilde{\Delta}^* (\mathscr{T}_{\mathcal{P}'}) $ is an isomorphism (using \cite[Ch. I, Proposition 3.5]{Milneetco} and \cite[Theorem C.15]{Sernesi}). Thus, the induced map \(\text{At}(\widetilde{\Delta})\) between the corresponding Atiyah sheaves is also isomorphism. Thus, we get the following commutative diagram of the corresponding Atiyah sequences

\begin{equation*}
	\xymatrixrowsep{1.8pc} \xymatrixcolsep{2.2pc}
	\xymatrix{
		0\ar@{->}[r]&
		\text{ad}(\mathcal{P})\ar@{->}[r]^{\iota}\ar@{->}[d]^{\text{ad}(\widetilde{\Delta})}& \mathcal{A}t(\mathcal{P})\ar@{->}[d]^{\text{At}(\widetilde{\Delta})}\ar@{->}[r]^{\eta}&\mathscr{T}_X\ar@{->}[d]_{Id}\ar@{->}[d]\ar@{->}[r]&0\\
		0\ar@{->}[r]&
		\text{ad}(\mathcal{P}')\ar@{->}[r]^{\iota'}&\mathcal{A}t(\mathcal{P}')\ar@{->}[r]^{\eta'}&\mathscr{T}_X\ar@{->}[r]&{0}.
	}
\end{equation*} 
This forces that the map \(\text{ad}(\widetilde{\Delta})\) is also an isomorphism. This shows that the Atiyah sequence for $\mathcal{P}$ and  \(\mathcal{P}':=\mathcal{P}_1 \times_X \mathcal{P}_2\) coincide. Hence, their \(G\)-Atiyah sequences also coincide. It remains to show that if the \(G\)-Atiyah sequences corresponding to $\mathcal{P}_1$ and $\mathcal{P}_2$ split, then the \(G\)-Atiyah sequence for $\mathcal{P}'$ also splits. We have the following fiber product diagram
\begin{equation*}
	\begin{tikzpicture}[description/.style={fill=white,inner sep=2pt}]
		\matrix (m) [matrix of math nodes, row sep=3em,
		column sep=2.5em, text height=1.5ex, text depth=0.25ex]
		{ \mathcal{P}':=\mathcal{P}_1 \times_X \mathcal{P}_2    & &   \mathcal{P}_1  \\
			\mathcal{P}_2    & &   X  \\ };
		\path[->]  (m-1-1) edge node[auto] {}(m-1-3);
		\path[ ->] (m-2-1) edge node[below] {}(m-2-3);
		\path[->] (m-1-1) edge node[auto] {}(m-2-1);
		\path[->] (m-1-1) edge node[above] {}(m-1-3);
		\path[ ->] (m-1-3) edge node[auto] {$p_1$} (m-2-3);
		\path[->] (m-2-1) edge node[above] {$p_2$} (m-2-3);
	\end{tikzpicture}	.
\end{equation*}
Let the Atiyah sequence for $\mathcal{P}_i$, \(i=1, 2\) be the following
\begin{equation*}
	0 \longrightarrow \text{ad}(\mathcal{P}_i)  \stackrel{\iota_i}\longrightarrow \mathcal{A}t(\mathcal{P}_i) 	\stackrel{\eta_i}\longrightarrow  \mathscr{T}_X  \longrightarrow 0.
\end{equation*}
And, denote the \(G\)-Atiyah sequence for $\mathcal{P}_i$, \(i=1, 2\) by
\begin{equation}\label{factorGAt}
0 \longrightarrow \text{ad}(\mathcal{P}_i) \stackrel{\iota^0_i} \longrightarrow  \mathcal{A}t_{\rho}(\mathcal{P}_i) \stackrel{pr_2} \longrightarrow \mathcal{V} \longrightarrow 0.
\end{equation}
Hence, by the functoriality of the Atiyah sequence (see \cite[Page 188]{At57}, \cite{H23}), we get that
\begin{equation*}
		0 \longrightarrow \text{ad}(\mathcal{P}_1) \oplus  \text{ad}(\mathcal{P}_2) \stackrel{\iota'}\longrightarrow \mathcal{A}t(\mathcal{P'}) 	\stackrel{\eta'=(\eta_1, \eta_2)}\longrightarrow  \mathscr{T}_X  \longrightarrow 0
\end{equation*}
is the Atiyah sequence for $\mathcal{P'}$, where
\begin{equation*}
	\mathcal{A}t(\mathcal{P'}) =\text{Ker}(\eta_1 \circ a_1 - \eta_2 \circ a_2) \subset \mathcal{A}t(\mathcal{P}_1) \oplus \mathcal{A}t(\mathcal{P}_2).
\end{equation*}
Here the map \(a_i\) denote the projection of \(\mathcal{A}t(\mathcal{P}_1) \oplus \mathcal{A}t(\mathcal{P}_2)\) to the \(i^{\text{th}}\) factor for \(i=1, 2\). Thus, the \(G\)-Atiyah sequence for the bundle $\mathcal{P}'$ is given by
\begin{equation}\label{GAtP'}
	0 \longrightarrow \text{ad}(\mathcal{P}_1) \oplus  \text{ad}(\mathcal{P}_2) \stackrel{\iota}\longrightarrow \mathcal{A}t_{\rho}(\mathcal{P'}) 	\stackrel{\text{pr}_{\mathcal{V}}}\longrightarrow  \mathcal V   \longrightarrow 0,
\end{equation}
where 
\begin{equation*}
	\mathcal{A}t_{\rho}(\mathcal{P'}) =\{(\delta_1, \delta_2, \delta) : (\eta_1 \circ a_1)(\delta_1)=(\eta_2 \circ a_2)(\delta_2)=\zeta(\delta) \} \subset \mathcal{A}t(\mathcal{P}_1) \oplus \mathcal{A}t(\mathcal{P}_2)\oplus \mathcal{V}.  
\end{equation*}
We are given that $\mathcal{P}_i$ admits a \(G\)-connection, i.e., a splitting
\begin{equation*}
	\lambda_i : \mathcal{V} \longrightarrow \mathcal{A}t_{\rho}(\mathcal{P}_i) 
\end{equation*}
of \eqref{factorGAt}. This will induce a splitting of the sequence \eqref{GAtP'}. Hence, $\mathcal{P'}$ admits a \(G\)-connection. This completes the proof.
\end{proof}

\begin{rmk}\label{exist_proj_rmk}
	As a corollary of the proof we get that the principal bundles $\mathcal{P}_1$ and $\mathcal{P}_2$ admit \(G\)-connection if and only if the fiber product $\mathcal{P}_1 \times_X \mathcal{P}_2$ admits \(G\)-connection.
\end{rmk}

Now we give a necessary and sufficient criterion for a principal \(H\)-bundle to admit \(G\)-connection (cf. \cite[Theorem 3.1]{AB03}).

\begin{prop}\label{Thm3.1AB3}
	Let \(H\) be a reductive linear algebraic group. With the notation as above, a principal \(H\)-bundle $\mathcal{P}$ admits a \(G\)-connection if and only if the following conditions hold:
	\begin{enumerate}
		\item the adjoint bundle $\text{ad}(\mathcal{P})$ admits a \(G\)-connection,
		\item for every character $\chi$ of \(H\), the line bundle $\mathcal{P} \times^{\chi} k$ associated to $\mathcal{P}$ admits a \(G\)-connection. 
	\end{enumerate}
\end{prop}

\begin{proof}
As \(H'\) is semisimple, Corollary \ref{exiat_ad} implies that $\mathcal{P}_1$ admits a \(G\)-connection if and only if \(\text{ad}(\mathcal{P}_1)\) admits a \(G\)-connection.	Note that we have the following direct sum decomposition
	\begin{equation*}
		\text{ad}(\mathcal{P}) \cong \text{ad}(\mathcal{P}_1) \oplus \text{ad}(\mathcal{P}_2)
	\end{equation*}
from the proof of Proposition \ref{exist_proj}. Also, note that the adjoint bundle \(\text{ad}(\mathcal{P}_2)\) is trivial as \(A\) is abelian.  Hence, it always admits a connection which implies it admits a \(G\)-connection also. Thus, $\mathcal{P}_1$ admits a \(G\)-connection if and only if \(\text{ad}(\mathcal{P})\) admits a \(G\)-connection by Remark \ref{exist_proj_rmk}. Next we determine the criterion for existence of \(G\)-connection on the associated \(A\)-bundle $\mathcal{P}_2$. Note that the abelian quotient \(A:=H / [H \, , \, H]\) is a product of copies of \(k^*\), say \(A \cong (k^*)^n\). Then the principal \(A\)-bundle $\mathcal{P}_2$ becomes a direct sum of \(n\) line bundles, say \(L_1 \oplus \ldots \oplus L_n\) (as \((k^*)^n\) acts faithfully on \(k^n\)). Hence,  by Remark \ref{exist_proj_rmk} $\mathcal{P}_2$ admits a \(G\)-connection if and only if each of the line bundles \(L_i\) admits so. Consider the character 
\begin{equation*}
	\chi_i :(k^*)^n \rightarrow k^* \text{ given by } (t_1, \ldots, t_n) \mapsto t_i.
\end{equation*}
Then, we have the isomorphism of line bundles 
\begin{equation*}
	L_i \cong \mathcal{P}_2 \times^{\chi_i} k.
\end{equation*}
Hence, the principal \(A\)-bundle $\mathcal{P}$ admits a \(G\)-connection if and only if the associated line bundles \(\mathcal{P}_2 \times^{\chi} k\) admits \(G\)-connection for all characters $\chi$ of \(A\). Note that since any character $\chi$ of \(H\) factors through the abelian quotient \(A\), say $\chi'$
\begin{equation*}
	\mathcal{P}_2 \times^{\chi'} k \cong \mathcal{P} \times^H A \times^{\chi'} k  \cong \mathcal{P} \times^{\chi} k.
\end{equation*}
Then we are done by Proposition \ref{exist_proj}.
\end{proof}

\subsection{Infinitesimal deformation map and connection}\label{infi_def}

Let \(X\) be a complete variety and $\mathcal{P}$ be a principal \(H\)-bundle on \(X\). Let $\rho : G \times X \rightarrow X$ be an action of a connected \text{algebraic group} \(G\) on \(X\). 

Let \(\widetilde{\mathcal{P}}:=\rho^* \mathcal{P} \rightarrow G \times X\) be the pull back of the principal \(H\)-bundle. Note that there is a canonical isomorphism 
\begin{equation*}
	\widetilde{\mathcal{P}}|_{\{1_G \} \times X} \cong P
\end{equation*}
as principal bundles over \(X\). Considering $\widetilde{\mathcal{P}}$ as a family of principal \(H\)-bundles on \(X\), we have the infinitesimal deformation map
\begin{equation}\label{def_map}
	\kappa_1 : \mathfrak{g} \longrightarrow H^1(X, \text{ad}(\mathcal{P})),
\end{equation}
which arises as follows: 

Consider the Atiyah sequence for the pull back principal \(H\)-bundle $\widetilde{\mathcal{P}}$ (see \eqref{Atseq})
\begin{equation}\label{infid_eq1}
	0 \longrightarrow \text{ad}(\widetilde{\mathcal{P}})  \stackrel{\widetilde{\iota}_0}\longrightarrow \mathcal{A}t(\widetilde{\mathcal{P}}) 	\stackrel{\widetilde{\eta}}\longrightarrow  \mathscr{T}_{G \times X}  \longrightarrow 0.
\end{equation}
Let \(p_1 : G \times X \rightarrow G\) and \(p_2 : G \times X \rightarrow X\) denote the corresponding projection maps, respectively. Then we have \(\mathscr{T}_{G \times X} \cong p_1^* \, \mathscr{T}_G \oplus p_2^* \, \mathscr{T}_X\). Hence, pulling back the exact sequence \eqref{infid_eq1} by the inclusion \(j:  p_1^* \, \mathscr{T}_G \hookrightarrow \mathscr{T}_{G \times X}\) we get the following commutating diagram with exact rows:
	\begin{equation}\label{infid_eq2}
	\xymatrixrowsep{1.8pc} \xymatrixcolsep{4pc}
	\xymatrix{
		0\ar@{->}[r]&
		\text{ad}(\widetilde{\mathcal{P}})\ar@{->}[r]^{\iota'_0}\ar@{->}[d]^{\text{Id}}& \mathcal{A}t(\widetilde{\mathcal{P}}) \times_{\mathscr{T}_{G \times X}} p_1^* \, \mathscr{T}_G
		\ar@{->}[d]^{}\ar@{->}[r]^{\eta'}&p_1^* \, \mathscr{T}_G \ar@{->}[d]_{j}\ar@{->}[d]\ar@{->}[r]&0\\
		0\ar@{->}[r]&
		\text{ad}(\widetilde{\mathcal{P}})\ar@{->}[r]^{\widetilde{\iota}_0}&\mathcal{A}t(\widetilde{\mathcal{P}})\ar@{->}[r]^{\widetilde{\eta}}&\mathscr{T}_{G \times X}\ar@{->}[r]&{0}.
	}
\end{equation}
We denote by \(\mathcal{A}t_j(\widetilde{\mathcal{P}}):=\mathcal{A}t(\widetilde{\mathcal{P}}) \times_{\mathscr{T}_{G \times X}} p_1^* \, \mathscr{T}_G\) for the notational convenience. Restricting the short exact sequence of the top row to \(\{1_G\} \times X\) and taking the connecting homomorphism we get the map $\kappa_1$ in \eqref{def_map} (see \cite[Section 3]{kahler_str}).
\begin{lemma}\label{def_map_lm}
	The connecting homomorphism 
	\begin{equation*}
		\kappa_2 : \mathfrak{g}=H^0(X, \mathcal{V}) \longrightarrow H^1(X, \text{ad}(\mathcal{P}))
	\end{equation*}
arising from the \(G\)-Atiyah sequence \eqref{G-Atseq} coincides with the homomorphism $\kappa_1$ defined in \eqref{def_map}. 
\end{lemma}

\begin{proof}
Let $\nu : X \cong \{1_G\} \times X \hookrightarrow G \times X$ denote the inclusion map given by \(x \mapsto (1_G, x)\).	We show that the restriction of the top exact sequence in \eqref{infid_eq2} via $\nu$ coincides with the \(G\)-Atiyah sequence \eqref{G-Atseq} for the principal bundle $\mathcal{P}$. This will immediately imply that $\kappa_1$ and $\kappa_2$ are also same, being the connecting homomorphism of the same sequence. First note that we have the following canonical isomorphisms
\begin{equation*}
	\begin{split}
		& \nu ^* (\text{ad}(\widetilde{\mathcal{P}})) \cong \text{ad}(\nu^* \widetilde{\mathcal{P}}) \cong \text{ad}(\mathcal{P}) \text{ since } \rho \circ \nu =\text{Id}_X  ~ (\text{see \cite[Lemma 1.5]{BMW12}} ),\\
		& \nu ^* (p_1^* \, \mathscr{T}_G) \cong \mathcal{O}_X \otimes_{k} \mathfrak{g} \text{ since } \mathscr{T}_G \cong \mathcal{O}_G \otimes_{k} \mathfrak{g} \text{ and } p_1 \circ \nu =1_G.
	\end{split}
\end{equation*}
By the functoriality of the Atiyah sequence (see \cite[Lemma 1.12]{BMW12}), the map $\rho$ induces the following commutating diagram.
	\begin{equation}\label{infid_eq3}
	\xymatrixrowsep{1.8pc} \xymatrixcolsep{4pc}
	\xymatrix{
		0\ar@{->}[r]&
		\text{ad}(\widetilde{\mathcal{P}})\ar@{->}[r]^{\widetilde{\iota}_0}\ar@{->}[d]^{\cong}& \mathcal{A}t(\widetilde{\mathcal{P}})
		\ar@{->}[d]^{}\ar@{->}[r]^{\widetilde{\eta}}& \mathscr{T}_{G \times X} \ar@{->}[d]_{d \rho}\ar@{->}[d]\ar@{->}[r]&0\\
		0\ar@{->}[r]&
	\rho^*	\text{ad}(\mathcal{P})\ar@{->}[r]^{}&\rho^* \mathcal{A}t(\mathcal{P})\ar@{->}[r]^{}& \rho^* \mathscr{T}_{ X}\ar@{->}[r]&{0}.
	}
\end{equation}
Here the left most vertical arrow is the canonical isomorphism using \cite[Lemma 1.5]{BMW12}. The bottom row is exact since the map $\rho$ is flat (see \cite[Remark 2.2.3]{BrStr}). Then using \cite[Chapter III, Lemma 1.3]{HShomAG}, we get that the right-hand square in \eqref{infid_eq3} is a pull-back diagram. Now pulling back the \eqref{infid_eq2} via $\nu$, we get
	\begin{equation}\label{infid_eq4}
	\xymatrixrowsep{1.8pc} \xymatrixcolsep{4pc}
	\xymatrix{
		0\ar@{->}[r]&
		\text{ad}(\mathcal{P})\ar@{->}[r]^{}\ar@{->}[d]^{\text{Id}}& \nu ^*\mathcal{A}t_j(\widetilde{\mathcal{P}}) 
		\ar@{->}[d]^{}\ar@{->}[r]^{\eta'}&\nu^* \, p_1^* \, \mathscr{T}_G \ar@{->}[d]_{\nu^*(j)}\ar@{->}[d]\ar@{->}[r]&0\\
		0\ar@{->}[r]&
		\text{ad}(\mathcal{P})\ar@{->}[r]^{}& \nu^* \mathcal{A}t(\widetilde{\mathcal{P}})\ar@{->}[r]^{}&\nu^*\mathscr{T}_{G \times X}\ar@{->}[r]&{0}.
	}
\end{equation}
Here the bottom row is exact as the map $j$ is injective. Also the right-hand square is a pull-back diagram as before. Finally pulling back \eqref{infid_eq3} via $\nu$, we get
	\begin{equation}\label{infid_eq5}
	\xymatrixrowsep{1.8pc} \xymatrixcolsep{4pc}
	\xymatrix{
		0\ar@{->}[r]&
		\text{ad}(\mathcal{P})\ar@{->}[r]^{}\ar@{->}[d]^{\cong}& \nu^* \mathcal{A}t(\widetilde{\mathcal{P}})
		\ar@{->}[d]^{}\ar@{->}[r]^{}& \nu^* \mathscr{T}_{G \times X} \ar@{->}[d]_{\nu^*(d \rho)}\ar@{->}[d]\ar@{->}[r]&0\\
		0\ar@{->}[r]&
			\text{ad}(\mathcal{P})\ar@{->}[r]^{}& \mathcal{A}t(\mathcal{P})\ar@{->}[r]^{}&  \mathscr{T}_{ X}\ar@{->}[r]&{0}.
	}
\end{equation}
Here we have used the fact that \(\rho \circ \nu =\text{Id}_X\). From the diagrams \eqref{infid_eq4} and \eqref{infid_eq5}, we see that the top row in \eqref{infid_eq4} is the pull-back of the bottom row of \eqref{infid_eq5} via the composition \(\nu^*(d \rho) \circ \nu^*(j)\).  Note that this map coincides with the map $\zeta$ in \eqref{funda_ac} by Remark \ref{DG-Lieq}. Hence, the lemma follows. 
\end{proof}

The following proposition indicates that infinitesimal deformations of a principal bundle can be characterized in terms of \(G\)-connections. When the underlying variety is smooth, this has been obtained in \cite[Section 2.3]{BPeqcn}.
\begin{prop}\label{def_map_prop}
	The principal \(H\)-bundle $\mathcal{P}$ admits a \(G\)-connection if and only if $\kappa_1$ is the zero map. 
\end{prop}

\begin{proof}
	Let \(\lambda : \mathcal{V} \rightarrow \mathcal{A}t_{\rho}(\mathcal{P})\) be a \(G\)-connection on $\mathcal{P}$. Consider the induced map 
	\begin{equation*}
		\lambda_* : H^0(X, \mathcal{V}) \longrightarrow H^0(X, \mathcal{A}t_{\rho}(\mathcal{P}))
	\end{equation*}
between the corresponding global sections. Let
\begin{equation}\label{pf_def_map_prop1}
	H^0(X, \mathcal{A}t_{\rho}(\mathcal{P})) \stackrel{\kappa_0} \longrightarrow H^0(X, \mathcal{V})  \stackrel{\kappa_2} \longrightarrow H^1(X, \text{ad}(\mathcal{P}))
\end{equation}
be the long exact sequence arising from the Atiyah sequence \eqref{G-Atseq}. Since \(\lambda\) is a connection, we have  
\(\kappa_0 \circ \lambda_* =Id_{H^0 (X, \mathcal{V})} \). This shows that \(\kappa_0\) is surjective. Hence, $\kappa_2=0$ from the exactness of \eqref{pf_def_map_prop1}. Finally, using Lemma \ref{def_map_lm} we get \(\kappa_1=0\).

Conversely, suppose that $\kappa_1$ is the zero map. Hence, using Lemma \ref{def_map_lm} and \eqref{pf_def_map_prop1} we get that the map $\kappa_0$ is surjective. Let \(V  \subset 	H^0(X, \mathcal{A}t_{\rho}(\mathcal{P}))\) be a subspace such that the restriction 
\begin{equation*}
	\kappa=\kappa_0|_V : V \stackrel{\cong} \longrightarrow \mathfrak{g}
\end{equation*}
is an isomorphism. Using this define a map
\begin{equation*}
	\underline{\lambda} : \mathfrak{g} \stackrel{\kappa^{-1}} \longrightarrow V \subset 	H^0(X, \mathcal{A}t_{\rho}(\mathcal{P})),
\end{equation*}
which further induces the following map
\begin{equation*}
	\lambda : \mathcal{V} \longrightarrow \mathcal{A}t_{\rho}(\mathcal{P}).
\end{equation*}
Note that \(\text{pr}_2 \circ \lambda =Id_{\mathcal{V}} \). Hence, \(\lambda\) is a \(G\)-connection on the principal bundle $\mathcal{P}$.
\end{proof}

	\section{Equivariant structure and G-connection}\label{Equivariant structure and Gconnection}
	
In this section we investigate the relationship between the condition that a principal bundle admits a $G$–equivariant structure
and the condition that it admits a (flat) $G$–connection. When the underlying variety is smooth, this has been studied in \cite{BSN15, BPeqcn}. 

Let \(X\) be a complete variety over $k$.	Let \(\mathcal{P}\) be a principal \(H\)-bundle on \(X\) and $\rho : G \times X \rightarrow X$ be an action of a connected \text{algebraic group} \(G\) on \(X\). An equivariant structure on $\mathcal{P}$ is a lift of the action $\rho$ on $\mathcal{P}$ given by a morphism ${\sigma}: G \times \mathcal{P} \rightarrow \mathcal{P}$ such that 
	\begin{itemize}
		\item the following diagram commutes,
	\begin{equation}\label{equi-str}
		\begin{tikzpicture}[description/.style={fill=white,inner sep=2pt}]
			\matrix (m) [matrix of math nodes, row sep=3em,
			column sep=2.5em, text height=1.5ex, text depth=0.25ex]
			{ G \times \mathcal{P}    & &   \mathcal{P}  \\
				G \times X    & &   X  \\ };
			\path[->]  (m-1-1) edge node[auto] {}(m-1-3);
			\path[ ->] (m-2-1) edge node[below] {}(m-2-3);
			\path[->] (m-1-1) edge node[auto] {$Id \times p$}(m-2-1);
			\path[->] (m-1-1) edge node[above] {${\sigma}$}(m-1-3);
			\path[ ->] (m-1-3) edge node[auto] {$p$} (m-2-3);
			\path[->] (m-2-1) edge node[above] {$\rho$} (m-2-3);
		\end{tikzpicture}	.
	\end{equation} 
\item the actions of \(H\) and \(G\) on $\mathcal{P}$ commute.

\end{itemize}
	
\begin{prop}\label{eqtoconn}
	If the principal \(H\)-bundle $\mathcal{P}$ admits an equivariant structure, then $\mathcal{P}$ admits an integrable \(G\)-connection.
\end{prop}	
	
\begin{proof}
	Let the principal bundle $\mathcal{P}$ admit an equivariant structure. The action $\sigma$ induces the following homomorphism of algebraic groups 
	\[\bar{\sigma} : G \longrightarrow  (\text{Aut}^H(\mathcal{P}))^0, \, g \mapsto (\sigma_g: z \mapsto \sigma(g, z)), \]
 since the action of \(G\) and \(H\) on $\mathcal{P}$ are compatible. This gives rise to the following map between Lie algebras
	\begin{equation*}
		d \bar{\sigma} : \mathfrak{g} \longrightarrow \text{Lie}(\text{Aut}^H(\mathcal{P})).
	\end{equation*}
	Composing it with the isomorphism \eqref{P2}, we get the map 
	\begin{equation*}
		\beta= \beta_{\mathcal{P}} \circ d \bar{\sigma} ~:~ \mathfrak{g} \longrightarrow H^0(\mathcal{P}, \mathscr{T}_{\mathcal{P}})^H \hookrightarrow H^0(\mathcal{P}, \mathscr{T}_{\mathcal{P}}).
	\end{equation*}
	This induces the following map as in \eqref{funda_ac}
	\begin{equation*}
		\widetilde{\zeta} : \mathcal{O}_{\mathcal{P}} \otimes_{k}  \mathfrak{g} \longrightarrow \mathscr{T}_{\mathcal{P}},
	\end{equation*}
	which is also compatible with the natural Lie algebra structures on the sheaves. Note that \( \widetilde{\mathcal{V}}:=\mathcal{O}_{\mathcal{P}} \otimes_{k}  \mathfrak{g}\) has a natural \(H\)-equivariant structure where \(H\) acts on $\mathfrak{g}$ trivially. Let \(V \subseteq \mathcal{P}\) be an \(H\)-invariant open subset. Then for \(\widetilde{f} \in \mathcal{O}_{\mathcal{P}}(V)\), $\delta \in \mathfrak{g}$ and \(h \in H\), we have
	
	\begin{equation*}
		\begin{split}
				\widetilde{\zeta}((\widetilde{f} \otimes \delta) \, h) & =	\widetilde{\zeta}(\widetilde{f} \otimes \delta\, h) =(\widetilde{f} \, h) \, \beta(\delta)|_V\\
				& = (\widetilde{f} \, \beta(\delta)|_V) \, h ~ (\text{since } \beta(\delta) \in H^0(\mathcal{P}, \mathscr{T}_{\mathcal{P}})^H)\\
				&= \widetilde{\zeta}((\widetilde{f} \otimes \delta)) \, h.
		\end{split}
	\end{equation*}
Hence, the map $\widetilde{\zeta}$ is also \(H\)-equivariant. Applying \(H\)-invariant pushforward \(p^H_*\), we get
	\begin{equation*}
		\zeta' : \mathcal{V} \cong p^H_* (\mathcal{O}_{\mathcal{P}}) \otimes_{k} \mathfrak{g} \longrightarrow p^H_* \mathscr{T}_{\mathcal{P}}=\mathcal{A}t(\mathcal{P}).
	\end{equation*}
	Now define the map \(\lambda\) as follows
	\begin{equation}\label{eqistr1.1} 
		\begin{split}
			\lambda : \mathcal{V} & \longrightarrow \mathcal{A}t(\mathcal{P}) \oplus \mathcal{V} \\
			 v & \longmapsto \zeta'(v) \oplus v.
			 \end{split}
	\end{equation}
We show that \(\eta \circ \zeta'=\zeta\). First note that as the bundle $\mathcal{P}$ is equivariant from \eqref{equi-str}, we get the following commutative diagram

\begin{equation}\label{equistr1.0}
	\begin{tikzpicture}
		\def\a{1.5} \def\b{2}
		\path
		(-\a,0) node (A) {$G$}      
		(\a,0) node (B) {$(\text{Aut}^H(\mathcal{P}))^0$}
		(0,-\b) node[align=center] (C) {$\text{Aut}^0(X)$};
		\begin{scope}[nodes={midway,scale=.75}]
			\draw[->] (A)--(B) node[above]{$\bar{\sigma}$};
			\draw[->] (A)--(C) node[left]{$\bar{\rho}$};
			\draw[->] (B)--(C) node[right]{$\theta$};
		\end{scope}
	\end{tikzpicture},
\end{equation}
	where that the map $\theta$ is defined in \eqref{gamma}. To see this, take \(g \in G\). We need to show that \(\theta(\sigma_g)=\rho_g\). But note that from \eqref{gammaeq}, we have
	\begin{equation*}
\begin{split}
			(\theta(\sigma_g))(p(z))= &p (\sigma_g(z))\\
			 = & \rho_g(p(z)) ~ (\text{using commutativity of the diagram in } \eqref{equi-str}).
\end{split}
	\end{equation*}
	Now the assertion follows using the surjectivity of the projection map \(p\). Taking the corresponding Lie algebras in \eqref{equistr1.0} and using \eqref{preli1}, we get the following commutative diagram of Lie algebras
	\begin{equation}\label{equistr1.2}
		\begin{tikzpicture}
			\def\a{1.5} \def\b{2}
			\path
			(-\a,0) node (A) {$\mathfrak{g}$}      
			(\a,0) node (B) {$H^0(\mathcal{P}, \mathscr{T}_{\mathcal{P}})^H$}
			(0,-\b) node[align=center] (C) {$H^0(X, \mathscr{T}_X)$};
			\begin{scope}[nodes={midway,scale=.75}]
				\draw[->] (A)--(B) node[above]{$\beta$};
				\draw[->] (A)--(C) node[left]{$\alpha$};
				\draw[->] (B)--(C) node[right]{$\eta(X)$};
			\end{scope}
		\end{tikzpicture}.
	\end{equation}
	Let \(U \) be an open subset of \(X\) and \(\delta \in \mathfrak{g}\). Denote by \(V=p^{-1}(U)\). Note that \(\zeta(1 \otimes \delta)= \alpha(\delta)|_U\) and $ \widetilde{\zeta}(1 \otimes \delta)=  \beta(\delta)|_V$. From \eqref{equistr1.2}, we have \((\eta(X) \circ \beta)(\delta)=\alpha(\delta)\). Since $\eta$ is a sheaf homomorphism, we get 
	\begin{equation*}
		\eta(\beta(\delta)|_V)= \eta(X) (\beta (\delta))|_{V} = \alpha(\delta)|_U.
	\end{equation*}
Hence, we get \(\eta \circ \zeta'=\zeta\).	Thus 
\begin{equation*}
	(\eta - \zeta)(\lambda(v))=(\eta \circ \zeta')(v)-\zeta (v)=0. 
\end{equation*}
Hence, we get that \(\text{Im}(\lambda) \subseteq \mathcal{A}t_{\rho}(\mathcal{P})\) and the induced map from \eqref{eqistr1.1} will also be denoted by 
	\begin{equation*}
		\lambda : \mathcal{V} \longrightarrow \mathcal{A}t_{\rho}(\mathcal{P}).
	\end{equation*}
	Clearly, the composition \(\text{pr}_2 \circ \lambda = Id_{\mathcal{V}}\). So \(\lambda\) defines a \(G\)-connection on the principal bundle $\mathcal{P}$. Also, since $\zeta'$ preserves the Lie algebra structure, \(\lambda\) is in fact an integrable connection. 
	\end{proof}

\begin{prop}\label{conntoint}
	Let the principal \(H\)-bundle \(p : \mathcal{P} \rightarrow X\) admit an integrable \(G\)-connection. Assume that \(G\) is semisimple and simply connected. Then $\mathcal{P}$ admits an equivariant structure.
\end{prop}

\begin{proof}
	Let \(\lambda : \mathcal{V} \rightarrow \mathcal{A}t_{\rho}(\mathcal{P})\) be an integrable \(G\)-connection. Since, \(\text{pr}_2 \circ \lambda = Id_{\mathcal{V}}\), we get
	\begin{equation*}
		\lambda( w)=(v,w),
	\end{equation*}
where \(v\), \(w\) are sections of the sheaves \(\mathcal{A}t(\mathcal{P})\) and $\mathcal{V}$ such that \((v, w) \) is a section of  $\mathcal{A}t_{\rho}(\mathcal{P})$, i.e.
\begin{equation}\label{conntoint1}
	\eta(v)=\zeta (w).
\end{equation}
 Consider the following composite of maps
\begin{equation*}
	\lambda' : \mathcal{V} \stackrel{\lambda} \longrightarrow \mathcal{A}t_{\rho}(\mathcal{P}) \stackrel{\text{pr}_1} \longrightarrow \mathcal{A}t(\mathcal{P}).
\end{equation*}
So \(\lambda'(w)=v\). Taking global sections we get the following map.
\begin{equation*}
	\lambda'(X) : \mathfrak{g} \longrightarrow H^0(\mathcal{P}, \mathscr{T}_{\mathcal{P}})^H.
\end{equation*}
Note that \(\lambda'(X)\) is a Lie algebra homomorphism as \(\lambda\) is integrable. Using \eqref{conntoint1}, we see that
\begin{equation}\label{conntoint2}
	\eta(X) \circ \lambda'(X)=\alpha.
\end{equation}
 Since \(G\) is connected, simply connected and semisimple for any algebraic group \(G'\) and any Lie algebra homomorphism from $\mathfrak{g}$ to \(\text{Lie}(G')\) lifts to a unique morphism from \(G\) to \(G'\) (see \cite[Proposition 2.9(b), Section 6, Ch II]{D-M}). Thus there are algebraic group homomorphisms \(\bar{\rho}= \bar{\rho'} : G \rightarrow \text{Aut}(X)\) and \(\bar{\sigma} : G \rightarrow \text{Aut}^H(\mathcal{P})\) such that
\begin{equation*}
	\text{Lie}(\bar{\rho}) = \alpha, \,  \text{Lie}(\bar{\rho'}) =\eta(X) \circ \lambda'(X) \text{ and } \text{Lie}(\bar{\sigma})=\lambda'(X).
\end{equation*}
Since \(\eta(X) \circ \lambda'(X)= \text{Lie}(\theta) \circ  \text{Lie}(\bar{\sigma})=\text{Lie}(\theta \circ \bar{\sigma})\), we have \( \bar{\rho'}=\theta \circ \bar{\sigma}\), i.e. the following diagram commutes.
\begin{equation*}
	\begin{tikzpicture}
		\def\a{1.5} \def\b{2}
		\path
		(-\a,0) node (A) {$G$}      
		(\a,0) node (B) {$\text{Aut}^H(\mathcal{P})$}
		(0,-\b) node[align=center] (C) {$\text{Aut}(X)$};
		\begin{scope}[nodes={midway,scale=.75}]
			\draw[->] (A)--(B) node[above]{$\bar{\sigma}$};
			\draw[->] (A)--(C) node[below]{$\bar{\rho}$};
			\draw[->] (B)--(C) node[right]{$\theta$};
		\end{scope}
	\end{tikzpicture} 
	\end{equation*}
(see \cite[Proposition 2.1, Section 6, Ch II]{D-M}). This implies that the conditions in \eqref{equi-str} are satisfied (cf. \cite[Remark 2.4(i)]{Br18}). Hence, $\mathcal{P}$ admits an equivariant structure.
\end{proof}

\begin{prop}
	Suppose the principal \(H\)-bundle \(p : \mathcal{P} \rightarrow X\) admits a \(G\)-connection. Assume that \(G\) is semisimple. Then, $\mathcal{P}$ admits an integrable \(G\)-connection.
\end{prop}
	
	\begin{proof}
		Let \(\lambda : \mathcal{V} \rightarrow \mathcal{A}t_{\rho}(\mathcal{P})\) be a \(G\)-connection. Then taking global sections, we get 
		\begin{equation*}
			\lambda(X) : \mathfrak{g} \longrightarrow H^0(X, \mathcal{A}t_{\rho}(\mathcal{P})),
		\end{equation*}
	which is $k$-linear. Since \(\lambda\) is a connection, we have \(\text{pr}_2(X) \circ \lambda(X) = Id_{\mathfrak{g}}\). Hence, the Lie algebra homomorphism \(\text{pr}_2(X)\) is also surjective. Since $\mathfrak{g}$ is semisimple there exists a Lie subalgebra $\mathfrak{a}$ of \(H^0(X, \mathcal{A}t_{\rho}(\mathcal{P}))\) such that 
	\begin{equation*}
		q:= \text{pr}_2(X)|_{\mathfrak{a}} : \mathfrak{a} \stackrel{\cong}\longrightarrow \mathfrak{g}
	\end{equation*}
is an isomorphism of Lie algebras (see \cite[Corollaire 3, pg. 91]{Bourbaki}). Consider the homomorphism of sheaves
\begin{equation*}
	\lambda': \mathcal{V} \longrightarrow \mathcal{A}t_{\rho}(\mathcal{P})
\end{equation*}
induced from the map \(q^{-1} : \mathfrak{g} \rightarrow \mathfrak{a} \subseteq H^0(X, \mathcal{A}t_{\rho}(\mathcal{P}))\). Clearly, \(\lambda'\) preserves the Lie algebra structures and \(\text{pr}_2 \circ \lambda' = Id_{\mathcal{V}}\) holds. Thus \(\lambda'\) is an integrable \(G\)-connection.
	\end{proof}
	
	\begin{cor}\label{exist_conn_1}
		Let \(p : \mathcal{P} \rightarrow X\) be a principal \(H\)-bundle. Assume that \(G\) is semisimple and simply connected. Then the following conditions are equivalent.
		\begin{enumerate}
			\item The principal bundle admits an equivariant structure.
			\item The principal bundle admits an integrable \(G\)-connection.
			\item The principal bundle admits a \(G\)-connection.
		\end{enumerate}
	\end{cor}
	
	\begin{rmk}{\rm 
		Some examples of semisimple and simply connected algebraic groups are \(SL(n, k)\), \(Sp(2n, k)\) and \(Spin(n, k)\) and non-examples are \(PGL(n, k)\) and \(SO(n, k)\) (see \cite[Table 9.2]{Malle}).
	}
	\end{rmk}
	
	The following proposition provides a criterion for existence of equivariant structure extending \cite[Corollary 4.4]{BPeqcn} to the case of varieties possibly having singularities.
	
	\begin{prop}\label{klytype}
		Let $p : \mathcal{P} \rightarrow X$ be a principal \(H\)-bundle on \(X\) such that for all \(g \in G\) we have isomorphisms
		\begin{equation}\label{eq_str_phi_g}
			\Phi_g : \mathcal{P} \stackrel{\cong} \longrightarrow \rho_g^* \, \mathcal{P}
		\end{equation}
	of principal \(H\)-bundles over \(X\). Assume that \(G\) is semisimple and simply connected. Then, $\mathcal{P}$ admits an equivariant structure.
	\end{prop}

\begin{proof}
	We first show that the given condition \eqref{eq_str_phi_g} implies that there is an isomorphism as principal \(H\)-bundles over \(G \times X\)
	\begin{equation*}
	\Psi:	p_2 ^* \, \mathcal{P} \stackrel{\cong} \longrightarrow \rho^* \, \mathcal{P}=: \widetilde{\mathcal{P}},
	\end{equation*}
	where \(p_2 : G \times X \rightarrow X\) denotes the projection onto the second factor. Note that
	\begin{equation*}
		\begin{split}
		&	p_2 ^* \, \mathcal{P} =\{(g, x, z) \in (G \times X \times \mathcal{P}) ~:~ x=p(z) \} \text{ and }	\widetilde{\mathcal{P}}= \{(g, x, z) \in (G \times X \times \mathcal{P}) ~:~ \rho(g,x)=p(z) \}.
		\end{split}
	\end{equation*}
Let \((g,x,z) \in p_2 ^* \, \mathcal{P}\). Then using the isomorphism $\Phi_g$, we get
\begin{equation*}
	\Phi_g(z)=(x, z') \text{ such that } \rho(g,x)=p(z').
\end{equation*}
Hence, we define \(\Psi(g,x,z)=(g,x,z')\). Note that its inverse map is given by
\begin{equation*}
	(g_1,x_1,z_1) \longmapsto (g_1,x_1, \Phi_{g_1}^{-1}(x_1,z_1)).
\end{equation*}
Thus the top exact sequence in \eqref{infid_eq2} admits a canonical splitting (cf. \cite[Section 3]{NR75}). Hence, by Lemma \ref{def_map_lm} the infinitesimal deformation map $\kappa_1$ is zero map. Then by Proposition \ref{def_map_prop} the principal bundle $\mathcal{P}$ admits a \(G\)-connection. Finally, using Corollary \ref{exist_conn_1} we conclude that $\mathcal{P}$ admits an equivariant structure.
\end{proof}
	
	\section{Examples}
	
	\subsection{Toric varieties}
	Let \(X\) be a complete toric variety under the action of the algebraic torus \(T \cong (k^*)^n\). Let \(D= X \setminus T\) be the boundary divisor. Then by \cite[Proposition 3.1]{Oda}, we have that the logarithmic tangent sheaf $\mathscr{T}_X(-\text{log}D) \subseteq \mathscr{T}_X$ is isomorphic to the trivial bundle \(\mathcal{O}_X \otimes_k \text{Lie}(T)\). Hence, a \(T\)-connection is same as a logarithmic connection singular along \(D\). Note that the torus \(T\) is not semisimple, so we cannot directly apply Corollary \ref{exist_conn_1}. A different proof is given in \cite{DKP}.

\subsection{Line bundle} Let \(X\) be a normal complete variety over \(k\). Let \(G\) be a connected linear algebraic group acting on \(X\). Then for any line bundle $\mathcal{L}$ on \(X\) there exists an integer \(n\) such that the tensor power $\mathcal{L}^{\otimes n}$ admits a \(G\)-equivariant structure (see \cite[Corollary 1.6]{GIT-MF}). Hence, $\mathcal{L}^{\otimes n}$ admits an integrable \(G\)-connection by Corollary \ref{exist_conn_1}. If \(X\) satisfies \(H^1(X, \mathcal{O}_X)=0\), then by Remark \ref{Gat_class} we see that any line bundle on \(X\) admits a \(G\)-connection. As a consequence, $\mathcal{O}_{\mathbb{P}^n}(1)$ admits a \(PGL(n+1, k)\)-connection but it does not admit any  \(PGL(n+1, k)\)-equivariant structure (see \cite[Page 33]{GIT-MF}). Note that \(PGL(n+1, k)\) is not simply connected (\(SL(n+1, k)\) being the simply connected cover). This shows that the assumptions on \(G\) in Corollary \ref{exist_conn_1} are necessary.

\subsection{Trivial Action}
If the action \(\rho : G \times X \rightarrow X\) is trivial, the image of the induced map \(\bar{\rho} : G \rightarrow \text{Aut}^0(X)\) is \(Id_X\). Hence, the induced map $\alpha : \mathfrak{g} \rightarrow H^0(X, \mathscr{T}_X)$ is the zero map. Hence, we have
\begin{equation*}
	\mathcal{A}t_{\rho}(\mathcal{P})=\text{ad}(\mathcal{P}) \oplus (\mathcal{O}_X \otimes \mathfrak{g}).
\end{equation*} 
Hence, a \(G\)-connection on the principal bundle $\mathcal{P}$ corresponds to a $\mathcal{O}_X$-module homomorphism \(\mathcal{O}_X \otimes \mathfrak{g} \rightarrow \text{ad}(\mathcal{P}) \).  Note that there is a tautological \(G\)-connection on $\mathcal{P}$ induced from the zero map from \(\mathcal{O}_X \otimes \mathfrak{g}\) to  \(\text{ad}(\mathcal{P}) \). Hence, by Proposition \eqref{eqtoconn} the principal bundle $\mathcal{P}$ admits a natural integrable \(G\)-connection.

\subsection{Tautological \(G\)-Connection}
Let \(X\) be a complete variety over \(k\) and \(p : \mathcal{P} \rightarrow X\) be a principal \(H\)-bundle. Recall that \(G:=(Aut^H(\mathcal{P}))^{\circ}\) is an algebraic group. Note that \(G\) has a tautological action on $\mathcal{P}$ and \(X\), respectively given as follows:
\begin{equation*}
	\begin{split}
	&	\rho_0: G \times \mathcal{P} \rightarrow \mathcal{P}, \text{ given by } (\psi, z) \rightarrow \psi(z)\\
		&\sigma_0: G \times X \rightarrow X, \text{ given by } (\psi, x) \rightarrow \theta(\psi)(x),
	\end{split}
\end{equation*}
where \(\psi \in G, x \in X\), \(z \in \mathcal{P}\) and $\theta$ is defined in \eqref{gamma}. Note that under these actions the principal bundle $\mathcal{P}$ becomes \(G\)-equivariant using \eqref{gammaeq}. Hence, it admits an integrable \(G\)-connection by Proposition \ref{eqtoconn} (cf. \cite[Proposition 3.3]{BPeqcn}).


\end{document}